\documentclass{article}

\usepackage{amsmath,amsthm,amssymb,amsfonts}

\usepackage{verbatim}
\usepackage{graphicx}

\usepackage{stmaryrd} %for \leftrightarroweq
\usepackage{hyperref}

\usepackage[colorinlistoftodos]{todonotes}

\newtheorem{thm}{Theorem}

\newtheorem{lem}[thm]{Lemma}

\newtheorem{prop}[thm]{Proposition}

\newtheorem{assert}[thm]{Assertion}

\newtheorem{remarks}[thm]{Remark}

\newtheorem{definition}[thm]{Definition}

\newtheorem{exl}[thm]{Example}

\numberwithin{thm}{section}

\newcommand{\adj}{\leftrightarrow}

\newcommand{\adjeq}{\leftrightarroweq}

\DeclareMathOperator{\id}{id}
\DeclareMathOperator{\Fix}{Fix}

\def\Z{{\mathbb Z}}
\def\N{{\mathbb N}}

\def\R{{\mathbb R}}

\begin{document}
\title{Remarks on Fixed Point Assertions in Digital Topology, 8}
\author{Laurence Boxer
\thanks{Department of Computer and Information Sciences, Niagara University, NY 14109, USA
and  \newline
Department of Computer Science and Engineering, State University of New York at Buffalo \newline
email: boxer@niagara.edu
\newline
ORCID: 0000-0001-7905-9643
}
}

\date{ }
\maketitle

\begin{abstract}
    This paper continues a series in which we study
    deficiencies in previously published works
    concerning fixed point assertions for digital images.


\end{abstract}

\section{Introduction}
There are many beautiful results concerning fixed points
for digital images. There are also many highly flawed
papers concerning this topic. The current work
continues that
of~\cite{BxSt19, Bx19, Bx19-3, Bx20, Bx22, BxBad6, BxBad7}
in discussing flaws in papers that have come to our
attention since acceptance of~\cite{BxBad7} for
publication.

In particular, the notion of
a ``digital metric space" has led many authors to attempt,
in most cases either erroneously or trivially, to modify fixed
point results for Euclidean spaces to digital images. 
This notion contains roots of all the
flawed papers studied in the current paper.
See~\cite{Bx20} for discussion of why ``digital metric space"
does not seem a worthy topic of further research.

\section{Preliminaries}
Much of the material in this section is quoted or
paraphrased from~\cite{Bx20}.

We use $\N$ to represent the natural numbers,
$\Z$ to represent the integers. 
%$\R$ to represent the reals, 
% and $\N^* = \N \cup \{0\}$.

A {\em digital image} is a pair $(X,\kappa)$, where $X \subset \Z^n$ 
for some positive integer $n$, and $\kappa$ is an adjacency relation on $X$. 
Thus, a digital image is a graph.
In order to model the ``real world," we usually take $X$ to be finite,
although there are several papers that consider
infinite digital images, e.g., for digital analogs of
covering spaces. The points of $X$ may be 
thought of as the ``black points" or foreground of a 
binary, monochrome ``digital picture," and the 
points of $\Z^n \setminus X$ as the ``white points"
or background of the digital picture.

\subsection{Adjacencies, 
%connectedness, 
continuity, fixed point}

In a digital image $(X,\kappa)$, if
$x,y \in X$, we use the notation
$x \adj_{\kappa}y$ to
mean $x$ and $y$ are $\kappa$-adjacent; we may write
$x \adj y$ when $\kappa$ can be understood. 
We write $x \adjeq_{\kappa}y$, or $x \adjeq y$
when $\kappa$ can be understood, to
mean 
$x \adj_{\kappa}y$ or $x=y$.

The most commonly used adjacencies in the study of digital images 
are the $c_u$ adjacencies. These are defined as follows.
\begin{definition}
\label{cu-adj-Def}
Let $X \subset \Z^n$. Let $u \in \Z$, $1 \le u \le n$. Let 
$x=(x_1, \ldots, x_n),~y=(y_1,\ldots,y_n) \in X$. Then $x \adj_{c_u} y$ if 
\begin{itemize}
    \item $x \neq y$,
    \item for at most $u$ distinct indices~$i$,
    $|x_i - y_i| = 1$, and
    \item for all indices $j$ such that $|x_j - y_j| \neq 1$ we have $x_j=y_j$.
\end{itemize}
\end{definition}

\begin{definition}
\label{path}
(See {\rm \cite{Khalimsky}}) 
    Let $(X,\kappa)$ be a digital image. Let
    $x,y \in X$. Suppose there is a set
    $P = \{x_i\}_{i=0}^n \subset X$ such that
$x=x_0$, $x_i \adj_{\kappa} x_{i+1}$ for
$0 \le i < n$, and $x_n=y$. Then $P$ is a
{\em $\kappa$-path} (or just a {\em path}
when $\kappa$ is understood) in $X$ from $x$ to $y$,
and $n$ is the {\em length} of this path.
\end{definition}

\begin{definition}
{\rm \cite{Rosenfeld}}
A digital image $(X,\kappa)$ is
{\em $\kappa$-connected}, or just {\em connected} when
$\kappa$ is understood, if given $x,y \in X$ there
is a $\kappa$-path in $X$ from $x$ to $y$.
\end{definition}

\begin{definition}
{\rm \cite{Rosenfeld, Bx99}}
Let $(X,\kappa)$ and $(Y,\lambda)$ be digital
images. A function $f: X \to Y$ is 
{\em $(\kappa,\lambda)$-continuous}, or
{\em $\kappa$-continuous} if $(X,\kappa)=(Y,\lambda)$, or
{\em digitally continuous} when $\kappa$ and
$\lambda$ are understood, if for every
$\kappa$-connected subset $X'$ of $X$,
$f(X')$ is a $\lambda$-connected subset of $Y$.
\end{definition}

\begin{thm}
{\rm \cite{Bx99}}
A function $f: X \to Y$ between digital images
$(X,\kappa)$ and $(Y,\lambda)$ is
$(\kappa,\lambda)$-continuous if and only if for
every $x,y \in X$, if $x \adj_{\kappa} y$ then
$f(x) \adjeq_{\lambda} f(y)$.
\end{thm}

\begin{remarks}
    For $x,y \in X$, $P = \{x_i\}_{i=0}^n \subset X$
is a $\kappa$-path from $x$ to $y$ if and only if
$f: [0,n]_{\Z} \to X$, given by $f(i)=x_i$, is
$(c_1,\kappa)$-continuous. Therefore, we may also
call such a function $f$ a {\em ($\kappa$-)path}
in $X$ from $x$ to $y$.
\end{remarks}

\begin{comment}
\begin{thm}
\label{composition}
{\rm \cite{Bx99}}
Let $f: (X, \kappa) \to (Y, \lambda)$ and
$g: (Y, \lambda) \to (Z, \mu)$ be continuous 
functions between digital images. Then
$g \circ f: (X, \kappa) \to (Z, \mu)$ is continuous.
\end{thm}
\end{comment}

We use $\id_X$ to denote the identity function on $X$, 
and $C(X,\kappa)$ for the set of functions 
$f: X \to X$ that are $\kappa$-continuous.

A {\em fixed point} of a function $f: X \to X$ 
is a point $x \in X$ such that $f(x) = x$. We denote by
$\Fix(f)$ the set of fixed points of $f: X \to X$.

Let $X = \Pi_{i=1}^n X_i$. The 
{\em projection to the
$j^{th}$ coordinate}
function $p_j: X \to X_j$
is the function defined
for $x = (x_1, \ldots, x_n) \in X$, $x_i \in X_i$, by $p_j(x) = x_j$.

As a convenience, if $x$ is a point in the
domain of a function $f$, we will often
abbreviate ``$f(x)$" as ``$fx$".

\subsection{Digital metric spaces}
\label{DigMetSp}
A {\em digital metric space}~\cite{EgeKaraca15} is a triple
$(X,d,\kappa)$, where $(X,\kappa)$ is a digital image and $d$ is a metric on $X$. The
metric is usually taken to be the Euclidean
metric or some other $\ell_p$ metric; 
alternately, $d$ might be taken to be the
shortest path metric. These are defined
as follows.
\begin{itemize}
    \item Given 
          $x = (x_1, \ldots, x_n) \in \Z^n$,
          $y = (y_1, \ldots, y_n) \in \Z^n$,
          $p > 0$, $d$ is the $\ell_p$ metric
          if \[ d(x,y) =
          \left ( \sum_{i=1}^n
          \mid x_i - y_i \mid ^ p
          \right ) ^ {1/p}. \]
          Note the special cases: if $p=1$ we
          have the {\em Manhattan metric}; if
          $p=2$ we have the 
          {\em Euclidean metric}.
    \item \cite{ChartTian} If $(X,\kappa)$ is a 
          connected digital image, 
          $d$ is the {\em shortest path metric}
          if for $x,y \in X$, $d(x,y)$ is the 
          length of a shortest
          $\kappa$-path in $X$ from $x$ to $y$.
\end{itemize}

We say a metric space $(X,d)$ is {\em uniformly discrete}
if there exists $\varepsilon > 0$ such that
$x,y \in X$ and $d(x,y) < \varepsilon$ implies $x=y$.

\begin{remarks}
\label{unifDiscrete}
If $X$ is finite or  
\begin{itemize}
\item {\rm \cite{Bx19-3}}
$d$ is an $\ell_p$ metric, or
\item $(X,\kappa)$ is connected and $d$ is 
the shortest path metric,
\end{itemize}
then $(X,d)$ is uniformly discrete.

For an example of a digital metric space
that is not uniformly discrete, see
Example~2.10 of~{\rm \cite{Bx20}}.
\end{remarks}

We say a sequence $\{x_n\}_{n=0}^{\infty}$ is 
{\em eventually constant} if for some $m>0$, 
$n>m$ implies $x_n=x_m$.
The notions of convergent sequence and complete digital metric space are often trivial, 
e.g., if the digital image is uniformly 
discrete, as noted in the following, a minor 
generalization of results 
of~\cite{HanBanach,BxSt19}.

\begin{prop}
\label{eventuallyConst}
{\rm \cite{Bx20}}
Let $(X,d)$ be a metric space. 
If $(X,d)$ is uniformly discrete,
then any Cauchy sequence in $X$
is eventually constant, and $(X,d)$ is a complete metric space.
\end{prop}

\section{Iterations
of~\cite{BotmartEtal}}
The paper~\cite{BotmartEtal}
is concerned with 
comparing the rates
of convergence of
convergent sequences
in a digital image,
perhaps especially
of sequences converging
to a fixed point of
a given function.
This is a problem
of greater theoretical than
practical interest,
as in the ``real
world," a digital
image is finite and,
usually, of small to
moderate size. Also,
the paper is flawed
as we discuss below.

\begin{itemize}
    \item In Definition~2.5 
of~\cite{BotmartEtal},
a digital metric
space $(E,\mu)$, a
function $T: E \to E$,
and a sequence 
$\{\alpha_n\}_{n=0}^\infty$, where
$0 \le \alpha_n \le 1$ are hypothesized.
Statement~(2.2) of
this definition calls for
\[ x_{n+1} =
   f_{T,\alpha_n}(x_n)
\]
where nothing appears
to define or 
describe~$f_{T,\alpha_n}$. Further,
the definition proceeds with a
subdefinition,
\[ \epsilon_n = 
\mu(x_{n+1}, f_{T,\alpha_n}(x_n))
\]
where $\mu$ is a metric,
so, by the above, we would have
$\epsilon_n = 0$.
It seems unlikely
that this is what
the authors intended.
\item The assertion of~\cite{BotmartEtal}
labeled as Theorem~4.1 uses ``$F(T)$" without
definition in its hypothesis. Its usage suggests
this is intended to be ``$F_T$", defined
earlier as the fixed point set of
the function $T$.
\item The argument
offered as proof of
Theorem~4.1 uses the
symbol~``$\delta$", 
seemingly as a 
nonnegative real number. As
this ``proof"
proceeds, it seems
that $\delta$ is
required to be
in the interval
$[0,1)$, but this
restriction is never
stated. What appears
to be the same
``$\delta$" appears
in the arguments
given as proofs of
the assertions labeled
Theorem~4.2 and
Theorem~4.3.
\item The paper's
assertion labeled as
Theorem~4.2 also 
depends on its
Definition~2.5,
which, as discussed 
above, is dubious.
\item The paper's
Example~4.4 
assumes the function
\[T(x) = x/2 + 3
\] 
is defined from $X$
to $X$, where $X$
is the set of
non-negative integers. 
Since, e.g.,
$T(1)=3.5$ is 
non-integral,
the formula given
for $T$ does not
belong in a discussion of
functions from~$X$
to~$X$.
\item Example~4.5 is
similarly flawed,
hypothesizing the
function 
\[ T(x) = \sqrt{x^2 - 8x + 40}
\]
as a function from~$X$ to~$X$ 
where~$X$ is the set
of nonnegative
integers. But, e.g.,
$T(1)=\sqrt{33}$ is
not an integer, so
the formula given
for $T$ does not
belong in a discussion of
functions from~$X$
to~$X$.
\end{itemize}

Perhaps appropriate rewriting can
yield valid results out of the assertions
of~\cite{BotmartEtal}. However, as written, 
none of the
``Main Results" of
this paper can be
regarded as all of
well defined, proven,
and valid.

\section{Contractive type mappings of~\cite{GuptaEtal}}
The paper~\cite{GuptaEtal} is concerned with
fixed point results
for self-maps $T$ on
digital images such that $T$ is a
$\theta$-contraction 
(defined below).

We note~\cite{GuptaEtal} uses
``$\ell$-adjacent" for what our
Definition~\ref{cu-adj-Def} would call
``$c_{\ell}-adjacent$".

\subsection{Improper citations}
Improper citations
exist
in~\cite{GuptaEtal}:
attributions to~\cite{HanBanach}
that should be
to~\cite{EgeKaraca15},
of the definition of
{\em digital metric space} and the 
definition of a {\em digital contraction
map} (not to be 
confused with
the ``contraction"
that is a digital
homotopy between an
identity map and a
constant map). Since~\cite{EgeKaraca15} is a reference
in~\cite{GuptaEtal},
the authors of the
latter should have
known better.

\subsection{\cite{GuptaEtal}'s Theorem~3.1}
\begin{definition}
    {\rm \cite{GuptaEtal}}
    $\Theta = \{
    \theta: [0,\infty) \to
    [0,\infty) \mid
    \theta$ is increasing,
    $\theta(t) < \sqrt{t}$ for $t>0$, $\theta(t)=0$
    if and only if
    $t=0 \}$.
\end{definition}

\begin{definition}
    {\rm \cite{GuptaEtal}}
    Suppose $(X,d,\ell)$ is
    a digital metric
    space, $T: X \to X$, and
    $\theta \in \Theta$. Suppose
    $d(Tx,Ty) \le
    \theta(d(x,y))$
    for all $x,y \in X$. Then $T$
    is called a
    $\theta${\em -digital
    contraction}.
\end{definition}

The following is Theorem~3.1
of~\cite{GuptaEtal}.

\begin{thm}
\label{Gupta3.1}
    Suppose $(X,d,\ell)$ is a digital metric
    space and $T: X \to X$ is
    a digital $\theta$-contraction for
    some $\theta \in \Theta$.
    Then $T$ has a unique fixed point.
\end{thm}

However, we note important cases for which the
previous theorem reduces to triviality, as a 
consequence of the following.

\begin{prop}
{\rm \cite{BxBad7}}
    Let $(X,d,\kappa)$ be a 
    connected digital metric space in which
    \begin{itemize}
        \item $d$ is the shortest path metric, or
        \item $d$ is any $\ell_p$ metric and
              $\kappa = c_1$.
    \end{itemize} 
 Then every $\theta$-contraction
    on $(X,d)$ is a constant map.
\end{prop}

\subsection{~\cite{GuptaEtal}'s Corollary 3.1}
Corollary~3.1
of~\cite{GuptaEtal}
gives a version of
the Banach Contraction Principle (defined below) for digital images. This
was previously
shown in~\cite{EgeKaraca-Ban}, which is a reference 
of~\cite{GuptaEtal},
so the authors of~\cite{GuptaEtal}
should have cited~\cite{EgeKaraca-Ban}.

Further, there are
important cases
for which the Banach
contraction principle for
digital images is a
triviality.

\begin{definition}
{\rm \cite{EgeKaraca-Ban}}
\label{digContractionMap}
    Let $(X,d,\kappa)$
    be a digital
    image. Let
    $0 \le \lambda < 1$ and suppose
    $T: X \to X$
    such that for all $x \in X$,
    $d(Tx,Ty) \le \lambda d(x,y)$. Then $T$
    is a {\em digital contraction map}.
\end{definition}

The Banach
contraction principle for
digital images is
the following.
\begin{thm}
    {\rm \cite{EgeKaraca-Ban,GuptaEtal}}
    Let $(X,d,\kappa)$
    be a digital
    image. Let
    $T: X \to X$
    be a digital 
    contraction map.
    Then $T$ has a
    unique fixed point.
\end{thm}

However, we have
the following.

\begin{thm}
\label{contractionTriv}
    Let $T: X \to X$
    be a contraction
    map on a connected digital image 
    $(X,d,\kappa)$.
    If $\kappa = c_1$ and
    $d$ is an $\ell_p$ 
    metric, or if $d$ is the
    shortest path
    metric, then
    $T$ is a constant map.
\end{thm}

\begin{proof}
For the case
$\kappa = c_1$ and
$d$ is an $\ell_p$ metric,
the assertion was
shown in~\cite{BxSt19}.

Now assume $d$ is the 
shortest path metric, and
let $\lambda$ be as in
Definition~\ref{digContractionMap}. Let $x \adjeq_{\kappa} y$ in~$X$.
Since $d(Tx,Ty) \le \lambda \cdot d(x,y) = \lambda$,
we must have $d(Tx,Ty) = 0$.
Thus $Tx=Ty$, and it follows
from the connectedness of
$(X,\kappa)$ that $T$ is
constant.
\end{proof}

Thus, for the cases discussed
in Theorem~\ref{contractionTriv}, the Banach contraction
principle for digital images
is a triviality.

\subsection{\cite{GuptaEtal}'s Theorem 3.2}
\label{secGupta3.2}
The assertion labeled Theorem~3.2 of~\cite{GuptaEtal} is
presented with a proof that
contains a major error. The
assertion is as follows.

\begin{assert}
{\rm \cite{GuptaEtal}}
\label{Gupta3.2}
    Let $(X,d,\ell)$ be a
    digital metric space and
    $T: X \to X$ such that
    $x \neq y$ implies
    $d(Tx,Ty) < \mu(x,y)$,
    where
    \[ \mu(x,y) = \max
    \left \{
    \begin{array}{c}
      \frac{1}{2} \left [
      d(y,Ty) \frac{1+d(x,Tx)}{1+d(x,y)} +
      d(Tx,Ty) + d(x,y) \right ] , \\
    d(x,Tx) \frac{1 + d(y,Ty)}{1+d(Tx,Ty)}
    \end{array}
    \right \} .
    \]
    Then $T$ has unique fixed point.
\end{assert}

In order to prove the
existence of a fixed point,
the authors define an infinite sequence of points
of~$X$ via $x_0 \in X$, 
$x_{n+1}=Tx_n$. They
attempt to show that if
the $x_n$ are all distinct 
then $\{d(x_n,x_{n+1})\}_{n=0}^{\infty}$ is a decreasing sequence. However, a chain
of equations and inequalities begins by claiming equality between
$d(x_n,x_{n+1})$ and
$d(x_{n-1},x_n)$. 
There is no obvious reason to
accept this alleged
equality, and note if true,
it would be counter to the goal of showing a decreasing sequence. The equation
should be
\[ d(x_n,x_{n+1})=
d(Tx_{n-1},Tx_n)\]
which would make the
next few lines of the
``proof" correct.

However, right side 
of the inequality
\[ \max \left \{
\begin{array}{c}
\frac{1}{2} \left [
d(x_n,x_{n+1})
\frac{1+d(x_{n-1},x_n)}{1+d(x_{n-1},x_n)} + d(x_n,x_{n+1}) +
d(x_{n-1},x_n)
\frac{1+d(x_n,x_{n+1})}{1+d(x_n,x_{n+1})}
\right ], \\ d(x_{n-1},x_n) \end{array}
\right \}
\]
\[ \le \max \left \{
\frac{1}{2}[d(x_n,x_{n+1}) +
 d(x_{n-1},x_{n+1})],
 d(x_{n-1},x_n)
\right \}
\]
should be
\[ \max \left \{
%\begin{array}{c}
\frac{1}{2} [ d(x_n,x_{n+1}) + 
d(x_n,x_{n+1}) +
d(x_{n-1},x_n)],
d(x_{n-1},x_n)
\right \} =
\]
\[ \max \{ [
   d(x_n,x_{n+1}) +
   \frac{1}{2}
   d(x_{n-1},x_n)],
d(x_{n-1},x_n) \}.
\]

Thus, the argument
fails to lead to the 
desired upper bound of
$d(x_{n-1},x_n)$
for $d(x_n,x_{n+1})$.
Therefore, the assertion
must be regarded as
unproven.

\section{\cite{Jain18}'s expansive maps}
D. Jain's paper~\cite{Jain18}
is concerned with
fixed point results
for expansive and
related digital maps.

The following definition is
presented without 
citation in~\cite{Jain18}. 
It should be attributed 
to~\cite{JyotiRani}.

\begin{definition}
    Suppose that $(X, d, \kappa)$ is a complete 
    digital metric space and 
    $S: X \to X$ is a mapping.
If $S$ satisfies 
$d(S(x), S(y)) \ge \alpha d(x, y)$ 
for all $x, y \in X$ and 
some $\alpha > 1$, 
then $S$ is a 
{\em digital expansive mapping}.
\end{definition}

\begin{remarks}
\label{discont}
It is easily derived
from a discussion
in~{\rm \cite{Rosenfeld}} -- see also~{\rm \cite{BxSt19}} --
that such a function
need not be 
digitally continuous. 
E.g., consider $S(x)=2x: \Z \to \Z$
with $\alpha = 1.5$,
$\kappa = c_1$. If
$x \adj_{c_1} y$ in~$X$
then $Sx \not \adjeq_{c_1} Sy$.
\end{remarks}

\subsection{\cite{Jain18}'s Theorem 3.2}
The following is
stated as Theorem~3.2
of~\cite{Jain18}.

\begin{assert}
\label{Jain3.2}
    Let $(X,d,\kappa)$
    be a digital metric space.
    Let $S$ be an
    onto continuous
    self map on $X$
    such that
    \[ d(Sx,Sy) \ge
    \alpha \mu(x,y)
    \]
    where $\alpha > 1$ and
    \[ \mu(x,y) =
    \max \left \{
    d(x,y), 
    \frac{d(x,Sx)+d(y,Sy)}{2},
    \frac{d(x,Sy)+d(y,Sx)}{2}
    \right \}.
    \]
    Then $S$ has a
    fixed point.
\end{assert}

This assertion is 
``almost" correct,
in that Jain's argument makes use of unstated
hypotheses, namely
that 
\begin{itemize}
    \item $(X,d,\kappa)$
is complete; and
\item $d$ is a metric (e.g., the
Euclidean metric) for which $x_n \to x$ implies that
for almost all~$n$,
$x_n = x$.
\end{itemize}
It may be that
Jain assumed that
$d$ is the Euclidean metric
(\cite{Jain18} 
does not specify $d$ as a particular metric), which satisfies
both of these
properties. An example of a metric
for a digital image
that fails to have
these properties is
given in Example~2.9
of~\cite{BxSt19}.

Also, the hypothesis of
continuity is both
unclear (Jain might
mean in the classic
$\varepsilon$ -- $\delta$ sense - Jain's probable intent, as hinted in the ``proof"; or
Jain might mean in 
the digital sense (which would be incorrect -- see
Remark~\ref{discont}) and unnecessary.
The argument offered 
in~\cite{Jain18} as proof of
Assertion~\ref{Jain3.2} should also be
corrected as discussed below.

Errors in Jain's 
``proof" are \begin{itemize}
    \item The inequality 
    ``$0 \ge k \mu(x,y)$" at line~6 of page~106 is stated without justification.
    It does not clearly follow from what precedes.
    \item Two lines later, we see the claim that
    ``$S$~is continuous", which had not been 
    hypothesized and we will not assume.
\end{itemize}

Theorem~\ref{correctedJain3.2}
below is a correct 
version of Assertion~\ref{Jain3.2}, in which we
omit the hypotheses
of~$d$ being complete and~$S$
being continuous.
Also, since
$\mu(x,y) \ge d(x,y)$, we
can substitute the latter 
for $\mu(x,y)$.

\begin{thm}
\label{correctedJain3.2}
        Let $(X,d,\kappa)$
    be a digital metric space, where
    $d$ is any
    $\ell_p$ metric
    or the shortest
    path metric.
    Let $S$ be an
    onto self map 
    on $X$ such that
    \[ d(Sx,Sy) \ge
    \alpha d(x,y)
    \]
    where $\alpha > 1$.
    Then $S$ has a
    fixed point.
\end{thm}

\begin{proof}
Note our choice 
of~$d$ makes
$(X,d,\kappa)$
complete, and
$d(x_n,x_0) \to 0$
implies $x_n=x_0$
for almost all~$n$.

We will modify Jain's argument
as is useful.
Jain's argument
starts with 
$x_0 \in X$. 
Since~$S$ is onto,
there exists 
$x_1 \in S^{-1}(x_0)$, and,
inductively, $x_n \in S^{-1}(x_{n-1})$. 

Jain correctly shows
$\{x_n\}_{n=0}^{\infty}$ is a Cauchy
sequence, although
this can be established 
much more briefly,
as follows.
\[ d(x_{n-1},x_n) =
d(Sx_n, Sx_{n+1})
\ge \alpha  
d(x_n,x_{n+1}).
\]
So
\[ d(x_n,x_{n+m}) \le
d(x_n,x_{n+1}) + 
\cdots + 
d(x_{n+m-1},x_{n+m}) \le
\]
\[
(1/\alpha)^n d(x_0,x_1) + \cdots
+ (1/\alpha)^{n+m-1}d(x_0,x_1) =
\]
\[ (1/\alpha)^n d(x_0,x_1) 
\left [
1 + (1/\alpha) + \cdots
+ (1/\alpha)^{m-1}
\right ]
\to_{n \to \infty} 0.
\]
Hence $\{x_n\}_{n=0}^{\infty}$ is a Cauchy
sequence.

By Remark~\ref{unifDiscrete} and 
Proposition~\ref{eventuallyConst}, there
exists $x \in X$
such that $x_n = x$
for almost all~$n$.
Therefore, for 
some~$n$,
$S(x_{n+1}) = x_n = x_{n+1}$, so
$x_n \in \Fix(S)$.
\end{proof}

We also note the
following.

\begin{thm}
    Let $(X,\kappa)$
    be a finite
    digital image
    of more than
    one point.
    For any metric~$d$, 
    there is no
    function~$S$
    that is as
    described
    in~Theorem~\ref{correctedJain3.2}.
\end{thm}

\begin{proof}
    Since $X$ is
    finite but has more than one point, there
    exist distinct 
    $x_0,x_1 \in X$ such
    that $d(x_0,x_1) = diam_d(X)$.
    Suppose there
    exists $S$ as
    described
    in~Theorem~\ref{correctedJain3.2}.
    Then
    \[ d(Sx_0,Sx_1)
    \ge \alpha \cdot \mu(x_0,x_1) \ge \alpha \cdot d(x_0,x_1) >
    d(x_0,x_1),
    \]
    contrary to our
    choice of $x_0,x_1$.
    The assertion 
    follows. 
\end{proof}

\subsection{\cite{Jain18}'s Theorem 3.3}
The following is stated as
Theorem~3.3 of~\cite{Jain18}.

\begin{assert}
    Let $(X,d,\kappa)$
    be a complete digital metric space and
    let $S: X \to X$ be
    an onto self map that is
    continuous. Suppose~$S$
    satisfies
    \[ d(Sx,Sy) \ge 
       \alpha \mu
    \]
    where $\alpha > 1$ and
    \[ \mu = \max \left \{ 
    d(x,y), 
    \frac{d(x,Sx)+d(y,Sy)}{2}, d(x,Sy), d(y,Sx)
    \right \}
    \]
    then $S$ has a fixed
    point.
\end{assert}

But notice that $\mu$ is greater than or equal to the
expression used for $\mu$ in Assertion~\ref{Jain3.2}. 
Therefore, we can make the same
modifications that we
made to Assertion~\ref{Jain3.2},
which gives us
Theorem~\ref{correctedJain3.2}.

\section{\cite{JyotiRani17}'s $\beta$--$\psi$ contractive maps}
The notion of a
``$\beta$--$\psi$ contractive map"
appears to have originated in~\cite{SametEtAl}
using the name ``$\alpha$--$\psi$ contractive map".
\cite{JyotiRani17} uses both 
``$\alpha$--$\psi$ contractive map" and
``$\beta$--$\psi$ contractive map" for the same notion.

\subsection{Fundamentals of $\beta$--$\psi$ contractive maps}
\begin{definition}
    {\rm \cite{JyotiRani17,SametEtAl}}
    $\Psi$ is the set of
    functions $\psi: [0, \infty) \to [0, \infty)$
    such that 

    i) $\psi$ is nondecreasing, and

    ii) there exist $k_0 \in \N$, $a \in (0,1)$, and a 
    convergent series of 
    nonnegative terms
    $\sum_{k=1}^{\infty} \nu_k$, such that
    \[ \psi^{k+1}(t) \le
       a \psi^k(t) + \nu_k
    \]
    for $k \ge k_0$ and all
    $t \in \R^+$.
\end{definition}

Note \cite{JyotiRani17} does not
define the symbol ``$\R^+$"; according
to~\cite{SametEtAl},
it represents
$[0,\infty)$.

\begin{lem}
    {\rm \cite{SametEtAl}}
    \label{PsiProps}
    $\psi \in \Psi$
    implies
    \begin{enumerate}
        \item For all $t \in \R^+$,
        $\psi^n(t) \to_{n \to \infty} 0$.
        \item $\psi(t) < t$
        for all $t>0$.
        \item $\psi$ is
        continuous at~0.
        \item $\sum_{n=1}^{\infty} \psi^n(t)$ converges, for
        all~$t \in \R^+$.
    \end{enumerate}
\end{lem}

\begin{definition}
\label{admissibleDef}
{\rm \cite{SametEtAl}}
    Let $T: X \to X$,
    $\alpha: X \times X \to [0,\infty)$. We say
    $T$ is {\em $\alpha$-admissible} if
    $\alpha(x,y) \ge 1$
    implies $\alpha(Tx,Ty) \ge 1$.
\end{definition}

\begin{definition}
    \label{alpha-psi-contractDef}
    Let $(X,d)$ be a metric
    space, $\alpha: X \times X \to [0,\infty)$,
    $\psi \in \Psi$. If
    $T: X \to X$ such
    that $x,y \in X$ implies
    \[ \alpha(x,y)d(Tx,Ty)
    \le \psi(d(x,y))
    \]
    then $T$ is an
    {\em $\alpha-\psi$ 
    contractive map}.
\end{definition}

\subsection{\cite{JyotiRani17}'s Theorem 3.3}
Theorem 3.3 of \cite{JyotiRani17} states
the following.
\begin{quote}
    Let $(X,d,\rho)$ be
    a complete digital metric space and let
    $T: X \to X$ be a
    $\beta-\psi$ contractive map
    such that
    \begin{enumerate}
        \item $T$ is
        $\beta$-admissible.
        \item There exists
        $x_0 \in X$ such
        that $\beta(x_0,Tx_0) \ge 1$.
        \item $T$ is digitally continuous.
    \end{enumerate}
    Then $T$ has a fixed
    point.
\end{quote}

The assertion is correct,
but, together with its
``proof," is substantially flawed, as follows.
\begin{itemize}
    \item We will show
    below that the assumptions of 
    completeness and continuity are
    unnecessary under ``usual" conditions.
    \item There are multiple incorrect references: The reference to~``(6)"
    should be to~``(1)" in
    the definition of
    $\beta-\psi$ contractive map; ``Using~(5)" should be
    ``Using~(4)"; and
    ``lemma~1.2" should be
    ``Lemma~2.2".
\end{itemize}

Thus, we can state the
following version of
Theorem~3.3 of~\cite{JyotiRani17}.

\begin{thm}
    \label{betterJyotiRani3.3}
    Let $(X,d, \rho)$ be
    a connected digital metric space,
    where $X$ is finite or $d$ is an
    $\ell_p$ metric or the
    shortest path metric. 
    Let $T: X \to X$ be a
    $\beta-\psi$ contractive map
    such that
    \begin{enumerate}
        \item $T$ is
        $\beta$-admissible.
        \item There exists
        $x_0 \in X$ such
        that $\beta(x_0,Tx_0) \ge 1$.
    \end{enumerate}
    Then $T$ has a fixed
    point.
\end{thm}

\begin{proof}
Note our assumptions 
about~$(X,d)$ imply 
completeness, by
Remark~\ref{unifDiscrete}
and Proposition~\ref{eventuallyConst}.

We use much of the argument
of~\cite{JyotiRani17}.
    Let $x_{n+1} = Tx_n$
    for all $n \ge 0$.
    By assumption, 
    \[ \beta(x_0,x_1) =
    \beta(x_0,Tx_0) \ge 1.
    \]
    A simple induction, using the fact 
    that~$T$ is $\beta$-admissible,
    lets us know
    \[ \beta(x_n,x_{n+1}) =
    \beta(Tx_{n-1},Tx_n) \ge 1 \mbox{ for }
    n > 0.
    \]
    Then $n>0$ implies
    \[ d(x_n,x_{n+1}) =
    d(Tx_{n-1},Tx_n) \le
    \beta(x_{n-1},x_n) \cdot
    d(Tx_{n-1},Tx_n) 
    \]
    \[ \le \psi(d(x_{n-1},x_n)).
    \]
    and a simple induction
    yields
    \[ d(x_n,x_{n+1}) \le
    \psi^n(d(x_0,x_1)).
    \]
By Lemma~\ref{PsiProps}(1), we have
$d(x_n,x_{n+1}) \to_{n \to \infty} 0$. By
Proposition~\ref{eventuallyConst}, for
almost all~$n$,
$x_n = x_{n+1} = T(x_n)$, so
$x_n \in \Fix(T)$.
\end{proof}

\subsection{\cite{JyotiRani17}'s Theorem 3.5}
Despite a reference
in the argument
offered as proof of
\cite{JyotiRani17}'s Theorem~3.5 to
a Theorem~3.4 in
the same paper,
there is no 
Theorem~3.4 
in~\cite{JyotiRani17}.

The following is
stated as Theorem~3.5
of~\cite{JyotiRani17}.

\begin{quote}
    Let $(X,d,\rho)$
    be a complete digital metric space. Let
    $T: X \to X$ be a digital
    $\beta-\psi$
    contractive map
    satisfying

    (i) $T$ is
    $\beta$-admissible;

    (ii) There exists $x_0 \in X$
    such that 
    $\beta(x_0,Tx_0) \ge 1$;

    (iii) If
    $\{x_n\}_{n=0}^{\infty} \subset X$ such that
    $\beta(x_n,x_{n+1}) \ge 1$ for
    all~$n$ and
    $x_n \to_{n \to \infty} x \in X$ then
    $\beta(x_n,x) \ge 1$ for
    almost all~$n$.

    Then there exists $u \in \Fix(T)$.
\end{quote}

The authors suggest
that the purpose of
this assertion is 
to be a version of
their Theorem~3.3
without the requirement of 
continuity. We note, however, that
our version of
their Theorem~3.3,
namely our Theorem~\ref{betterJyotiRani3.3},
requires neither a
continuity assumption nor item~(iii) of the current assertion.

\subsection{\cite{JyotiRani17}'s Examples 3.6 and 3.7}
The authors wish 
to demonstrate uses
of their previous
assertions in
these examples.
However, the 
metric spaces
considered are
not digital
images, as they
are not subsets
of any~$\Z^n$.

\section{\cite{Khan}'s fixed point
assertion for pairs of functions}

\cite{Khan}
uses $\Psi$ as the
symbol for a  function, not
the set of functions so labeled
in~\cite{JyotiRani17}. 
In~\cite{Khan},
$\Psi: [0,\infty) \to [0,\infty)$ is
continuous, and
$\Psi(t) = 0$ if
and only if $t=0$.

The author also assumes a
function $\phi: [0,\infty) \to [0,\infty)$ that is
lower semi-continuous such 
that $\Psi(t) = 0$ if
and only if $t=0$. It seems
likely that the latter is
meant to be $\phi(t) = 0$ if
and only if $t=0$.

\subsection{``Theorem" 3.1}
The following is
stated as Theorem~3.1
of~\cite{Khan}.

\begin{assert}
    Let $(X,d,\rho)$
    be a complete digital metric space.
    Let $N$ be a nonempty closed subset of $X$.
    Let $P,Q: N \to N$,
    $G,H: N \to X$ such that
    $Q(N) \subset H(N)$ and for all $x,y \in X$,
    \begin{equation}
    \label{Khan3.1ineq}
 \Psi(d(Px,Qy)) \ge
      \phi(d_{G,H}(x,y)) +
      \frac{1}{2}\Psi(d_{G,H}(x,y) + \phi(d_{G,H}(x,y))
    \end{equation}
    where
    \begin{equation} 
   \label{dGH}
  %  \[
    d_{G,H}(x,y)) = \max
    \left \{ \begin{array}{c}
    d(x,y), d(Gx,Hy), d(Gx,Px), 
    d(Hy,Qy), \\
    \frac{1}{3}d((Gx,Qy) + (Hy, Px))
    \end{array} 
    \right \} 
%\]
\end{equation}
and
\begin{equation}
    \label{dPQ}
    d_{P,Q}(x,y) =
    \max \left \{
    \begin{array}{c}
     d(x,y), d(Gx,Hy), d(Gx,Px), 
    d(Hy,Qy),    \\
    \frac{1}{4}d((Gx,Qy) +     (Hy,Px)) 
    \end{array}
    \right \}
%\]
\end{equation}
Then $\{P,G\}$
and $\{Q,H\}$ have a unique point of coincidence 
in~$X$. Moreover, if
$\{P,G\}$ and
$\{Q,H\}$ are self-mappings,
then $P,Q,G$,
and~$H$ have a
unique fixed point in~$X$.
\end{assert}

This assertion
has the following deficiencies.
\begin{itemize}
\item If~$d$ is a
``usual" metric, $(X,d)$ is a discrete topological space, so all subsets of~$X$ are closed.
Either the requirement of~$N$
being closed is
unnecessary, or the
author meant something else.
\item The statement of
    the assertion has
    $Q(N) \subset H(N)$, but
    the first line of the ``proof" says $Q(N) \subset G(N)$ and
    $P(N) \subset H(N)$.
    \item Statements~(\ref{Khan3.1ineq}), (\ref{dGH}),
    and~(\ref{dPQ}) are
    expected to be true for all~$x,y \in X$,
    but if $N \neq X$ then $Gx,Hx,Px,Qx$
    are undefined for~$x \in X \setminus N$.
   \item In each of~(\ref{dGH}) and~(\ref{dPQ}),
    it seems likely that ``$d((Gx,Qy) + (Hy,Px))$"
    should be ``$d(Gx,Qy) + d(Hy,Px)$".
    \item $d_{P,Q}$, defined in~(\ref{dPQ}), is not
    mentioned in~(\ref{Khan3.1ineq}). 
    \begin{itemize} 
       \item It seems that either $d_{P,Q}$ is intended 
             to be part of~(\ref{Khan3.1ineq}) or else it 
             should be deleted.
       \item It appears that correction of the second lines
             of both~(\ref{dGH}) and~(\ref{dPQ}) would leave
             $d_{P,Q}(x,y) \le d_{G,H}(x,y)$ for all
             $x,y \in X$. Therefore, perhaps~$d_{P<Q}$
             is unnecessary.
    \end{itemize}    
    \item What does it mean
    to say ``$\{P, G\}$ and $\{Q, H\}$ have a unique point of
    coincidence"? This notion is undefined in~\cite{Khan}. 
    Does it mean $P,G,Q$, and~$H$ have a 
    unique point of coincidence?
    \item It is claimed in the
    ``proof" that
    \[ \Psi(y_{2n+1},y_{2n+2})=\Psi(Px_{2n},Qx_{2n+1}) \le
    \]
    \[ \Psi(d_{G,H}(x_{2k},x_{2k+1}))
    - \phi(d_{G,H}(x_{2k},x_{2k+1})).
    \]
    It appears the symbol ``$d$" is missing from the left
    side of this inequailty. Further, there is no explanation 
    of this statement, and there is no obvious derivation of 
    it from~(\ref{Khan3.1ineq}) even after the seemingly
    appropriate corrections.
\end{itemize}
We conclude that
whatever this assertion is meant to
say, is unproven.

\subsection{Example 3.2}
Example 3.2 of~\cite{Khan} is
meant to illustrate
the paper's ``Theorem"~3.1,
discussed above.
The example uses
$X=[4,40]$, which
is not a subset of
any $\Z^n$ and thus
cannot be the set
underlying a digital image.
Also, $H$ is undefined
for some members
of~$X$ and
multiply defined
for others; e.g.,
$H(9)$ is undefined,
while $H(x)$ is defined both by
$17+x$ and by $16$
for $13 \le x \le 14$.

\subsection{Corollary 3.3}
The following is
stated as 
Corollary~3.3
of~\cite{Khan}.
Despite being labeled a corollary,
the assertion has
no clear relation
to previous 
assertions in~\cite{Khan},
and, in fact, is
false.

\begin{assert}
\label{Khan3.3}
    Let $P$ and $Q$ be self 
    mappings of a complete 
    digital metric space 
    $(X, d, \rho)$ into itself. 
    Suppose $P(X) \subset Q(X)$.
    If there exists
    $\alpha \in (0,1)$
    and a positive integer~$k$ such that
    $d(P^k(x),P^k(y)) \le \alpha d(Qx,Qy)$ for
    all $x,y \in X$,
    then $P$ and
    $Q$ have a unique common
    fixed point.
\end{assert}

To show Assertion~\ref{Khan3.3} is false,
consider the following.

\begin{exl}
    Let $X = [0,1]_{\Z}$,
    $d(x,y) = |x-y|$. Let $P(x)=0$,
    $Q(x)=1-x$.
    Clearly, 
    $P(X) \subset Q(X)$, and for all $x,y \in X$,
    \[ d(P^1(x),P^1(y)) = 0 \le 0.5 d(Qx,Qy)
    \]
    but~$Q$
    has no fixed
    point.
\end{exl}

\section{\cite{MishraEtal}'s common fixed points}
Theorem~5 of~\cite{MishraEtal}
says the following.
\begin{thm}
    Consider two commuting and self-mappings~$f$ 
    and~$g$ on a complete digital metric space
    $(K,d,p)$  with coefficient $\alpha \in (0,1)$ such
    that $f$ is continuous, 
    $g(K) \subset f(K)$, and
    \[ d(gx,gy) \le  \alpha d(fx,fy)
    ~~~~\mbox{ for all } x,y \in K.
    \]
    Then $f$ and $g$ have a unique
    common fixed point in~$K$.
\end{thm}

However, this is unoriginal; indeed,
stronger results appear in older papers.
Theorem~3.1.4 of~\cite{RaniEtal} 
shows the
assertion above can be part 
of an ``if and only if" theorem.
Theorem~5.3 of~\cite{BxFPsets2}
improves on Theorem~3.1.4 of~\cite{RaniEtal}
by showing that the
hypotheses of completeness and continuity
are unnecessary, and the ``commuting" condition can
be replaced by a weaker restriction of ``weakly commuting".

\section{\cite{Pal}'s fixed point theorems}
\subsection{\cite{Pal}'s Theorem 3.1}
\begin{definition}
    {\rm \cite{GuptaEtal}}
    Let $\Theta$ be
    the set of     functions
    $\theta: [0,\infty) \to [0,\infty)$
    such that $\theta$ is
    increasing,
    $\theta(t) < \sqrt(t)$ for
    $t > 0$, and
    $\theta(t)=0$
    if and only if $t=0$.
\end{definition}

Theorem~3.1 of~\cite{Pal}
states the following.

\begin{thm}
\label{Pal3.1}
    Let $(X,d,\beta)$ be a digital
    metric space, $T: X \to X$,
    and $\theta \in \Theta$.
    Suppose 
    \begin{equation}
    \label{Pal3.1ineq}
        d(Tx,Ty) \le \theta(d(x,y)) \mbox{ for all
    $x,y \in X$.}
    \end{equation}  
    Then $T$ has a unique fixed point.
\end{thm}

Theorem~\ref{Pal3.1} is, some would say, correctly
proven in~\cite{Pal}; others would say that where
the argument offered as proof reaches the
correct inequality 
\[ d(x_{n+1},x_n) < d(x_n, x_{n-1})\]
and claims that this shows $x_n \in \Fix(T)$ for almost
all~$n$, that this conclusion should be established by
an easy but absent argument using~(\ref{Pal3.1ineq}).

We observe important cases
for which Theorem~\ref{Pal3.1} reduces to
triviality.

\begin{prop}
    Let $X$ and $T$ be as in
    Theorem~\ref{Pal3.1}. Suppose $(X,d,\beta)$
    is connected.
    If $d$ is the
    shortest path
    metric, or if
    $\beta=c_1$ and $d$ is any
    $\ell_p$ metric, then
    $T$ is a constant function.
\end{prop}

\begin{proof}
    Given $x \adj y$ in $(X,\beta)$,
    we have 
    \[ d(Tx,Ty) \le \theta(d(x,y))= \theta(1) < \sqrt{1} =1   
    \]
    so $d(Tx,Ty)=0$,
    i.e., $Tx=Ty$.
    The assertion
    follows from
    connectedness.
\end{proof}

\subsection{\cite{Pal}'s Theorem 3.2}
\cite{Pal} states its ``Theorem"~3.2 as
another attempt to
obtain what we have
called Assertion~\ref{Gupta3.2}. The argument, and
its errors, are similar to those
of~\cite{GuptaEtal}.

Like~\cite{GuptaEtal},
\cite{Pal}'s attempt
to show $\{d(x_n,x_{n+1}) \}$
is a decreasing sequence has a chain
of comparisons beginning with 
an incorrect claim
that $d(x_n,x_{n+1})$
is equal to 
$d(x_{n-1},x_n)$,
where it should say
$d(x_n,x_{n+1}) =
d(Tx_{n-1},Tx_n)$.

As in~\cite{GuptaEtal},
correcting this error
lets us proceed through subsequent
lines of the argument,
until we come to the
claimed inequality
\[ \max \left \{
\begin{array}{c}
\frac{1}{2} \left [
d(x_n,x_{n+1})
\frac{1+d(x_{n-1},x_n)}{1+d(x_{n-1},x_n)} + d(x_n,x_{n+1}) +
d(x_{n-1},x_n)
\frac{1+d(x_n,x_{n+1})}{1+d(x_n,x_{n+1})}
\right ], \\ d(x_{n-1},x_n) \end{array}
\right \}
\]
\[ \le \max \left \{
\frac{1}{2}[d(x_n,x_{n+1}) +
 d(x_{n-1},x_{n+1})],
 d(x_{n-1},x_n)
\right \},
\]
an error that appeared in the
``proof" of~\cite{GuptaEtal},
discussed above in section~\ref{secGupta3.2},
correction of which does not lead
to the desired conclusion.

\section{\cite{SalJh}'s Dass-Gupta contraction}
\cite{SalJh} studies a digital version of
the Dass-Gupta contraction~\cite{DassGup}.

\subsection{\cite{SalJh}'s Theorem 3}
Theorem~3 of \cite{SalJh}, as written, is not
correctly proven, although with a small
number of minor changes, a correct result can be obtained. 
The following is stated as Theorem~3 of \cite{SalJh}.

\begin{assert}
   \label{SalJh3}
   Let $(F,\Phi,\gamma)$ be a complete digital metric space.
   Let $K: F \to F$ be a mapping that satisfies the 
   rational contraction condition
   \[ \Phi(Ku,Kv) \le \frac{\xi_1 \Phi(v,Kv)[1 + \Phi(u,Ku)]}{1 + \Phi(u,v)} + \xi_2 \Phi(u,v) \mbox{ for all }
   u,v \in F,
   \]
   where $\xi_1, \xi_2 > 0$ and $\xi_1 + \xi_2 < 1$.
   Then~$K$ has a unique fixed point.
\end{assert}

The argument offered as proof of Assertion~\ref{SalJh3}
in~\cite{SalJh} is marred by the following.
For $u_0 \in F$, the sequence $\{u_n = Ku_{n-1}\}$ 
and the constant $\eta = \frac{\xi_2}{1 - \xi_1}$
are defined. Then, in the argument for existence of a fixed
point:
\begin{itemize}
    \item The inequality
    \[ \Phi (u_n, u_{n+k}) \le \frac{\eta ^ n}{1-\eta} \Phi(u_0,u_1)
    \]
    is derived. It is argued that this inequality 
    shows $\{ u_n \}$ is a Cauchy sequence. But this
    line of reasoning requires $\eta < 1$, hence
    $\xi_2 < 1 - \xi_1$, which is not hypothesized.
    \item It is claimed that $K$ is continuous. This was
    neither hypothesized nor proven, and, we show below,
    is unnecessary.
\end{itemize}
And in the argument for the uniqueness of a fixed point:
\begin{itemize}
    \item If $\mu$ and $\lambda$ are fixed points,
          ``$\Phi(\mu, \lambda) = (K \mu, K \lambda)$"
          should be
          ``$\Phi(\mu, \lambda) = \Phi(K \mu, K \lambda)$".
    \item ``$0 < \lambda < 1$" should be
          ``$0 < \xi_2 < 1$".
\end{itemize}
Therefore, as written, Assertion~\ref{SalJh3} is unproven.

We modify Assertion~\ref{SalJh3} and its ``proof"
to obtain the following.

\begin{thm}
    {\rm \cite{SalJh}}
    Let $(F,\Phi,\Gamma)$
    be a digital metric
    space, and
    $K: F \to F$ a
    mapping satisfying
    \begin{equation}
        \label{SalJh(1)}
        \Phi(Ku,Kv) \le
        \frac{\xi_1 \Phi(v,Kv) [1+ \Phi(u,Ku)]}{1+\Phi(u,v)} +
        \xi_2 \Phi(u,v)
    \end{equation}
    for all $u,v \in F$, where
    $\xi_1, \xi_2 > 0$,
    $\xi_1 + \xi_2 < 1$,
    and
    \begin{equation}
        \label{SalJhAddon}
        \frac{\xi_2}{1-\xi_1} < 1.
    \end{equation}
    Then $K$ has a unique fixed point.
\end{thm}

Note the inequality~(\ref{SalJhAddon})
does not appear in~\cite{SalJh}, but
seems to be necessary in the proof.

\begin{proof}
    We use ideas of~\cite{SalJh}, with modifications
    to correct and
    abbreviate the argument
    as are desirable.

    Let $u_0 \in F$,
    $u_n = Ku_{n-1}$
    for all $n \in \N$.
        
    Then
    \[ \Phi(u_n,u_{n+1}) =
    \Phi(Ku_{n-1},Ku_n) \le \]
    \[ \frac{\xi_1 \Phi(u_n,Ku_n) [1+\Phi(u_{n-1},Ku_{n-1})]}{1+\Phi(u_{n-1},u_n)} +
    \xi_2 \Phi(u_{n-1},u_n) =
    \]
    \[ \frac{\xi_1 \Phi(u_n,u_{n+1})[1+\Phi(u_{n-1},u_n)]}{1+\Phi(u_{n-1},u_n)}+
    \xi_2 \Phi(u_{n-1},u_n) =
    \]
    \[ \xi_1 \Phi(u_n,u_{n+1}) +
    \xi_2 \Phi(u_{n-1},u_n).
    \]
    Thus
    \[ \Phi(u_n,u_{n+1}) \le \frac{\xi_2}{1-\xi_1} \Phi(u_{n-1},u_n).
    \]
An easy induction allows us to conclude that
 \[ \Phi(u_n,u_{n+1}) \le \left (\frac{\xi_2}{1-\xi_1}
 \right ) ^n \Phi(u_0,u_1).
    \]

So either $u_{n+1}=u_n$, in which case
$u_n \in \Fix(K)$; or, (by (\ref{SalJhAddon})), $\{\Phi(u_n,u_{n+1})\}$ is a sequence decreasing to 0,
so $u_n=u_{n+1}$ for $n \ge n_0$, for some
$n_0 \in \N$. Thus $u_{n_0} \in \Fix(K)$.

To show the uniqueness
of our fixed point,
suppose $u,v \in \Fix(K)$. By~(\ref{SalJh(1)}),
\[ \Phi(u,v) = \Phi(Ku,Kv) \le 
0 + \xi_2 \Phi(u,v).
\]
So $\Phi(u,v)=0$, i.e.,
$u=v$.
\end{proof}

\subsection{\cite{SalJh}'s Example 3.2}
    This example uses
    $F=[0,1]$, which is
    not a subset of~$\Z^n$ for any~$n$. Thus~$F$
    is not appropriate
    for use as a digital image.

\section{\cite{Shukla}'s assertions with rational inequalities}
The paper~\cite{Shukla}
claims fixed point 
results for self maps
satisfying certain
rational inequalities
on digital images.
However, these assertions are not well
defined.

The assertion labeled
``Theorem"~3.1 of~\cite{Shukla} is as
follows.

\begin{assert}
\label{ShuklaThm}
    Let $(X,d,\kappa)$
    be a digital metric
    space. Let $S,T$
    be self maps on~$X$
    satisfying
\[        \begin{array}{c}
        d(S(x,y),T(u,v)) \le \\ \\
        \alpha [d(x,u)+d(y,v)]+
        \beta [d(x,S(x,y))+d(u,T(u,v)] ~+ 
        \\ \\
        \gamma [ d(x,T(u,v)) + d(u,S(x,y)) +
        \delta \left ( \frac{ d(x,S(x,y)) d(u,T(u,v))}{d(x,u)+d(y,u)}
        \right )
        \\ \\
        + \eta \left [
        \frac{[d(x,u) + d(y,u)] [d(x,S(x,y)) + d(u,T(u,v))]}{1+d(x,u)+d(y,v)}
        \right ] +
       % \\ \\
               \end{array}
       \]
\begin{equation}
\label{ShuklaIneq}
        \zeta \left [
        \frac{d(u,S(x,y)+ d(x,T(u,v))}{1+d(u,T(u,v)) d(u,S(x,y)}
        \right ]
    \end{equation}
for all $u,v,x,y \in X$
and 
\[ 2(\alpha + \beta + \eta) + 4(\gamma + \zeta) + \delta < 1.
\]
Then $S$ and $T$ have a
common fixed point.
\end{assert}

Several items are
undefined in inequality~(\ref{ShuklaIneq}):
\begin{itemize}
    \item $S$ and $T$
    are supposed to be
    defined on~$X$, but
    throughout the inequality they appear as if defined
    on~$X \times X$.
    \item The line beginning ``$\gamma [$" is missing a
    matching ``$]$" --
    should it come after
    the $d(u,S(x,y))$
    term, or perhaps at
    the end of the line?
    \item In the same
    line, the last term
    has a denominator 
    of~0 when $x=u=y$.
    \item Shouldn't the
    coefficients $\alpha, \beta, \gamma, \delta, \eta, \zeta$ have
    lower bounds (perhaps 0)?
\end{itemize}

Also, the ``proof" of
this ``theorem" begins
with sequences described
as follows: $x_0,y_0$
are arbitrary members
of~$X$. Then
\begin{quote}
``$x_{n+1}= S(x_n,y_n)$, $y_{n+1}=T(y_n,x_n)$,
and $x_{n+2}=T(x_{n+1},y_{n+1})$, $y_{n+2}=S(y_{n+1},x_{n+1})$ for $n \in \N$."
\end{quote}
Accordingly, we might have
$x_{n+2}=T(x_{n+1},y_{n+1})$, or we might have
$x_{n+2}=x_{(n+1)+1}=S(x_{n+1},y_{n+1})$. Similarly, $y_{n+2}$
appears defined in two
ways that seem incompatible (unless $S=T$).

We also observe that the ``Corollaries"~3.2
through 3.6 of~\cite{Shukla} all share one or more of the
flaws discussed above.
We conclude that whatever is intended
by each of Assertion~\ref{ShuklaThm} and its ``Corollaries"
in~\cite{Shukla} is unproven.

\section{Further remarks}
We have discussed several papers that seek to advance
fixed point assertions for digital metric spaces.
Many of these assertions are incorrect, incorrectly proven, 
or reduce to triviality; consequently, all of these
papers should have been rejected or required to undergo
extensive revisions. This reflects badly not only on
the authors, but also on the referees and editors that
approved their publication.

\end{document}